\documentclass[10pt]{amsart}
\usepackage{amssymb}

\newtheorem{thm}{Theorem}[section]
\newtheorem{lemma}[thm]{Lemma}

\newtheorem{definition}[thm]{Definition}

\newcommand{\q}{\mathbb{Q}}
\newcommand{\R}{\mathbb{R}}

\newcommand{\cf}{\mathrm{cf}}
\newcommand{\cof}{\mathrm{cof}}
\newcommand{\dom}{\mathrm{dom}}

\newcommand{\add}{\textrm{Add}}

\newcommand{\Succ}{\mathrm{Succ}}

\newcommand{\h}{\mathrm{ht}}

\newcommand{\res}{\upharpoonright}

\newcommand{\qk}{\q_\kappa}
\newcommand{\qo}{<_{\q_\kappa}}
\newcommand{\rk}{\R_\kappa}
\newcommand{\ro}{<_{\R_\kappa}}
\newcommand{\ka}{\kappa}
\newcommand{\kp}{\kappa^+}

\begin{document}

\title{A Large Pairwise Far Family of Aronszajn Trees}

\author{John Krueger}

\address{John Krueger \\ Department of Mathematics \\ 
	University of North Texas \\
	1155 Union Circle \#311430 \\
	Denton, TX 76203}
\email{jkrueger@unt.edu}

\date{February 2021; Revised September 2022}



\thanks{This material is based upon work supported by the Simons Foundation under Grant 631279}

\begin{abstract}
We construct a large family of normal $\ka$-complete $\rk$-embeddable non-special 
$\kp$-Aronszajn trees which 
have no club-isomorphic subtrees using an instance of the proxy principle of 
Brodsky-Rinot \cite{rinot22}. 
\end{abstract}

\maketitle

Two trees of the same height 
are said to be \emph{club isomorphic} if there exists a club subset of their height 
and an isomorphism between the trees restricted to that club. 
Abraham-Shelah \cite{AS} proved a number of essential results about 
club isomorphisms of $\omega_1$-Aronszajn trees. 
They showed that under \textsf{PFA}, any two normal $\omega_1$-Aronszajn trees are 
club isomorphic. 
In the other direction, they proved that the weak diamond principle on $\omega_1$ implies 
the existence of a family of $2^{\omega_1}$ 
many pairwise non-club-isomorphic $\omega_1$-Aronszajn trees. 
A natural problem is to generalize these theorems to higher Aronszajn trees. 
Krueger \cite{jk31} 
constructed a model in which any two normal countably closed $\omega_2$-Aronszajn trees are club isomorphic. 
Chavez-Krueger \cite{chavez} proved that for any regular uncountable cardinal 
$\ka$ satisfying $\ka^{<\ka} = \ka$ and $\Diamond(\kp \cap \cof(\ka))$, 
there exists a pairwise far family of $2^{(\kp)}$ many special $\kp$-Aronszajn 
trees, meaning that 
any two trees in the family have no club-isomorphic subtrees. 
The proof of the latter result used a technique of sealing potential club isomorphisms using diamond.

In the present article we build a pairwise far 
family of $2^{(\kp)}$ many normal $\ka$-complete 
$\rk$-embeddable $\kp$-Aronszajn trees, where 
$\ka$ is a regular uncountable cardinal satisfying $\ka^{<\ka} = \ka$. 
There are two main differences between this construction and that of \cite{chavez}. 
First, instead of using diamond, we build the trees using 
the proxy principle of Brodsky-Rinot \cite{rinot22}. 
This principle implies the existence of 
a useful combinatorial object which is a combination of a club guessing sequence 
and a diamond sequence, and has been used to construct a variety of 
Aronszajn and Suslin trees (\cite{rinot22}, \cite{rinot23}, \cite{rinot29}). 
Secondly, rather than sealing potential club isomorphisms, we create a more explicit 
distinction between the different trees by ensuring the non-existence of stationary 
antichains below certain stationary sets. 
In particular, in contrast to the main result of \cite{chavez}, the trees in our 
collection are non-special.

In Section 1 we provide the necessary information about the linear orders $\qk$ and $\rk$ which 
we use to build our trees, review a standard construction of a higher Aronszajn tree, 
and describe the material on the proxy principle which we will need. 
In Section 2 we show how the proxy principle can be used to produce two 
non-club-isomorphic Aronszajn trees; the general idea behind this result, which is repeated 
in Section 3, is to arrange the destruction of certain types of potential 
stationary antichains. 
In Section 3 we present the main result of the article, which is the construction of a large 
pairwise far family 
of Aronszajn trees which generalizes the proof of Section 2. 

I would like to thank Assaf Rinot for his helpful and extensive 
comments on an earlier draft of this article, and also 
the referee for many helpful suggestions and corrections.

\section{Aronszajn trees and the proxy principle}

We assume that the reader is familiar with the basic definitions and facts about trees, although 
we will carefully review the notation used in this article. 
Let $(T,<_T)$ be a tree. 
A \emph{chain} is a linearly ordered subset of $T$, and an \emph{antichain} 
is a set of pairwise incomparable elements of $T$. 
A \emph{branch} is a maximal chain. 
For each $x \in T$, 
let $\h_T(x)$ denote the height of $x$ in $T$. 
For each ordinal $\delta$, let 
$T(\delta) := \{ x \in T : \h_T(x) = \delta \}$ denote level $\delta$ of $T$, and 
$T \restriction \delta := \{ x \in T : \h_T(x) < \delta \}$. 
More generally, if $A$ is a subset of the height of a tree $T$, 
$T \restriction A := \{ x \in T : \h_T(x) \in A \}$. 
A branch $b$ of $T$ is \emph{cofinal} if $b \cap T(\delta) \ne \emptyset$ for all $\delta$ less 
than the height of $T$. 
For an infinite cardinal $\ka$, $T$ is \emph{$\ka$-complete} if every chain of $T$ whose order 
type is less than $\ka$ has an upper bound in $T$.

Let $\lambda$ be a regular uncountable cardinal. 
A \emph{$\lambda$-tree} is a tree of height $\lambda$ such that for all $\delta < \lambda$, 
$T(\delta)$ has size less than $\lambda$. 
A $\lambda$-tree is \emph{Aronszajn} if it has no cofinal branch and is 
\emph{Suslin} if it has no chains or antichains of size $\lambda$. 
If $\lambda = \mu^+$ is a successor cardinal, a tree $T$ of height $\lambda$  
is \emph{special} if there exists a function 
from $T$ to $\mu$ which is injective on chains. 
Note that a $\lambda$-tree being special implies that it is Aronszajn and not Suslin.
The next result can be proved by a simple pressing down argument.

\begin{lemma}
Let $\lambda = \mu^+$ be a successor cardinal and $T$ a $\lambda$-tree. 
If $S \subseteq \lambda$ is stationary and $T \restriction S$ is special, 
then there exists an antichain $A \subseteq T \restriction S$ 
such that the set $\{ \h_T(x) : x \in A \}$ is a stationary subset of $S$.
\end{lemma}

For a $\lambda$-tree $T$, a set $A \subseteq T$ is a \emph{stationary antichain} 
if $A$ is an antichain and the set $\{ \h_T(x) : x \in A \}$ is stationary in $\lambda$. 
If $S \subseteq \lambda$, $A$ is a \emph{stationary antichain below $S$} if it is 
a stationary antichain and $\{ \h_T(x) : x \in A \} \subseteq S$.

A \emph{subtree} of a tree $T$ is any subset $U$ of $T$ considered as a tree with the 
inherited order $<_T \cap \ (U \times U)$. 
A set $X \subseteq T$ is \emph{downwards closed} if for all $x \in X$, 
$\{ y \in T : y <_T x \} \subseteq X$. 
Note that branches are downwards closed.
If $U$ is a downwards closed subtree of $T$, then for all $x \in U$, $\h_T(x) = \h_U(x)$. 
For any set $Y \subseteq T$, the \emph{downward closure} of $Y$ is the set 
$\{ z \in T : \exists y \in Y \ z \le_T y \}$. 

A $\lambda$-tree $T$ is \emph{normal} if:
\begin{enumerate}
\item $T$ has a unique element of height $0$, which is called the \emph{root} of $T$;
\item for every $x \in T$ and every $\gamma < \lambda$ above $\h_T(x)$, 
there exists $y \in T$ such that $x <_T y$ and $\h_T(y) = \gamma$;
\item if $x$ and $y$ are distinct nodes of $T$ 
with the same 
limit height $\delta$, 
then the sets $\{ z \in T : z <_T x \}$ and $\{ z \in T : z <_T y \}$ are not equal;
\item for every node $x$ of $T$, there are incomparable nodes $y$ and $z$ above $x$.
\end{enumerate}

Let $T$ and $U$ be $\lambda$-trees. 
A function $f : T \to U$ is an \emph{isomorphism} if 
$f$ is a bijection and for all $x$ and $y$ in $T$, 
$x <_T y$ iff $f(x) <_U f(y)$. 
We say that $T$ and $U$ are \emph{isomorphic} if there 
exists an isomorphism from $T$ to $U$. 
A \emph{club isomorphism} of $T$ and $U$ is an isomorphism 
$f : T \restriction C \to U \restriction C$, where $C \subseteq \lambda$ is some club. 
If there exists a club isomorphism of $T$ and $U$, then $T$ and $U$ are \emph{club isomorphic}. 
It is easy to verify that if $f : T \restriction C \to U \restriction C$ 
is a club isomorphism of $T$ and $U$, 
then for all $x \in T \restriction C$, 
$\h_T(x) = \h_U(f(x))$. 
Define $T$ and $U$ to be \emph{near} if there exist downwards closed subtrees of $T$ and $U$ 
which are club-isomorphic $\lambda$-trees, 
and otherwise $T$ and $U$ are \emph{far}.

In order to construct Aronszajn trees, we will employ generalized versions of the rationals 
and the reals (see \cite[Section 3]{Honzik} for more complete details).

\begin{definition}
	Let $\ka$ be a regular cardinal. 
	Define $\qk$ as the set of all functions $f : \ka \to 2$ such that 
	the set $\{ \alpha < \ka : f(\alpha) = 1 \}$ is non-empty and has size 
	less than $\ka$, ordered lexicographically.
\end{definition}

Observe that $\qk$ is a linear order of size $2^{<\ka}$ which is dense in itself 
and without endpoints. 
If $\ka^{<\ka} = \ka$, then $\qk$ has size $\ka$.

We list the basic properties of $\qk$ which we will use in our Aronszajn tree constructions.

\begin{lemma}
	Let $\ka$ be a regular cardinal.
	\begin{enumerate}
		\item Between any two elements of $\qk$ there exists an increasing sequence 
		of order type $\ka$;
		\item any increasing sequence of elements of 
		$\qk$ with order type less than $\ka$ has a 
		least upper bound in $\qk$;
		\item any decreasing sequence of elements of 
		$\qk$ whose order type is a limit ordinal 
		less than $\ka$ does not have a greatest lower bound.
	\end{enumerate}
\end{lemma}

The proof of (1) is routine. 
(2) and (3) are proven in \cite[Lemma 3.4]{Honzik}.

\begin{definition}
	Let $\ka$ be a regular cardinal. 
	Define $\R_\ka$ as the Dedekind completion of $\qk$.
\end{definition}

Note that $\R_\ka$ is a dense complete linear order without endpoints in which 
$\q_\ka$ is dense.

\begin{lemma}
	Let $\ka$ be a regular cardinal. 
	\begin{enumerate}
		\item Between any two elements of $\rk$ there exists an increasing sequence 
		of order type $\ka$;
		\item any increasing sequence of elements of $\rk$ 
		with order type a limit ordinal less than $\ka$ has a 
		least upper bound in $\qk$.
	\end{enumerate}
\end{lemma}

\begin{proof}
	(1) Consider $r_0 \ro r_1$. By the density of $\qk$, we can fix 
	$r_0 \ro q_0 \ro q_1 \ro r_1$ 
	where $q_0$ and $q_1$ are in $\qk$. 
	Now apply Lemma 1.3(1) to $q_0$ and $q_1$.
	
	(2) Suppose that $\langle r_i : i < \delta \rangle$ is an increasing sequence in $\rk$, 
	where $\delta < \ka$ is a limit ordinal. 
	By the density of $\qk$, for each $i < \delta$ we can choose $q_i \in \qk$ 
	such that $r_i \ro q_i \ro r_{i+1}$. 
	Applying Lemma 1.3(2), let $q$ be the least upper bound of $\{ q_i : i < \delta \}$ in $\q_\ka$. 
	Then $q$ is also the least upper bound of $\{ r_i : i < \delta \}$ in $\R_\ka$.
\end{proof}

\begin{definition}
	Let $T$ be a tree and $L$ a linear order. 
	We say that $T$ is \emph{$L$-embeddable} if there exists a function 
	$f : T \to L$ such that $x <_T y$ implies $f(x) <_L f(y)$ for all $x, y \in T$. 
\end{definition}

Suppose that $\ka$ is a regular cardinal such that $\ka^{<\ka} = \ka$. 
Then $\qk$ has size $\ka$. 
Hence, if $T$ is a $\kp$-tree which is $\qk$-embeddable, then $T$ is special 
(the converse is also true by \cite[Lemma 3.12]{Honzik}). 
Also, $T$ being $\rk$-embeddable implies that $T$ is Aronszajn and not Suslin. 
This fact is well-known when $\ka = \omega$. 
For $\ka > \omega$, note that $T \restriction (\kp \setminus \cof(\ka))$ 
is special by Lemma 1.5, and then apply Lemma 1.1.

We now review a standard construction of a normal 
$\ka$-complete special $\kp$-Aronszajn tree 
which uses the linear order $\qk$. 
This construction will be a blueprint for more complicated constructions given later 
in the article, so we will provide many of the details.

Fix a regular cardinal $\ka$ and assume that $\ka^{<\ka} = \ka$. 
We will define by recursion a $\kp$-tree $T$ together with a map $\pi : T \to \qk$. 
We will maintain several properties of $T$ and $\pi$:
\begin{enumerate}
	\item $T(0)$ consists of $0$, $T(1)$ consists of the ordinals in the interval 
	$[1,\ka \cdot 2)$, and for each $1 < \alpha < \kp$, $T(\alpha)$ consists of 
	the ordinals in the interval $[\ka \cdot \alpha, \ka \cdot (\alpha+1) )$; 
	\item $x <_T y$ implies $\pi(x) \qo \pi(y)$, for all $x, y \in T$;
	\item for each $x \in T$, the restriction of $\pi$ to the immediate successors of $x$ 
	is a bijection from that set onto the set $\{ q \in \qk : \pi(x) \qo q \}$;
	\item if $\delta \in \kp \cap \cof(<\!\ka)$, then 
	every cofinal branch $b$ of $T \restriction \delta$ has a unique 
	upper bound $y$ in $T(\delta)$, and $\pi(y) = \sup \{ \pi(x) : x \in b \}$;
	\item for all $x \in T$, $\beta < \kp$, and $q \in \qk$ such that 
	$\h_T(x) < \beta$ and $\pi(x) \qo q$, 
	there exists $y \in T(\beta)$ above $x$ satisfying 
	$\pi(y) \le_{\qk} q$. 
\end{enumerate}

We will abbreviate the restriction of $\pi$ to $T \restriction \alpha$ by $\pi \restriction \alpha$, 
for all $\alpha < \kp$. 

For the base case, let $T(0)$ consist of $0$, and define $\pi(0)$ to be some 
arbitrary element of $\qk$. 
Let $T(1)$ consist of $[1,\ka \cdot 2)$, and define $\pi$ on $T(1)$ to be 
some bijection between $T(1)$ and the set $\{ q \in \qk : \pi(0) \qo q \}$.

For the successor case, 
let $0 < \alpha < \kp$ and assume that $T \restriction (\alpha+1)$ and 
$\pi \restriction (\alpha+1)$ are defined. 
Let the elements of $T(\alpha+1)$ consist of the ordinals in 
$[\ka \cdot (\alpha+1), \ka \cdot (\alpha+2))$. 
Let each node of $T(\alpha)$ have exactly $\ka$ many immediate successors in $T(\alpha+1)$. 
Define $\pi$ on $T(\alpha+1)$ so that for each $x \in T(\alpha)$, 
$\pi$ maps the set of immediate successors of $x$ onto the set 
$\{ q \in \qk : \pi(x) \qo q \}$. 
The inductive hypotheses are easy to check.

For the limit case, 
let $\delta < \kp$ be a limit ordinal and assume that 
$T \restriction \delta$ and $\pi \restriction \delta$ are defined. 
First, assume that $\delta$ has cofinality less than $\ka$. 
Since $\ka^{<\ka} = \ka$ and $\cf(\delta) < \ka$, 
there are exactly $\ka$ many cofinal branches of $T \restriction \delta$. 
Let the nodes of $T(\delta)$ be the ordinals in $[\ka \cdot \delta,\ka \cdot (\delta+1) )$. 
Place exactly one node above each cofinal branch of $T \restriction \delta$. 
For each $x \in T(\delta)$, define $\pi(x) := \sup \{ \pi(y) : y <_T x \}$.

Let us verify the inductive hypotheses. 
(1), (2), (3), and (4) are immediate. 
For (5), consider any $x \in T \restriction \delta$ and $q \in \qk$ such that $\pi(x) \qo q$. 
Applying Lemma 1.3(1), fix an increasing sequence $\langle q_i : i < \cf(\delta) \rangle$ 
of elements of $\qk$ between $\pi(x)$ and $q$. 
Fix an increasing and continuous sequence $\langle \delta_i : i < \cf(\delta) \rangle$ 
of ordinals cofinal in $\delta$ 
such that $\h_T(x) < \delta_0$. 
Using inductive hypotheses (4) and (5), 
recursively build a chain $\langle x_i : i < \cf(\delta) \rangle$ above $x$ such that 
$x_i \in T(\delta_i)$ and $\pi(x_i) \le_{\qk} q_i$ for all $i < \cf(\delta)$.
Let $y$ be an upper bound of this chain in $T(\delta)$. 
Then $x <_T y$ and $\pi(y) \le_{\qk} q$.

Secondly, assume that $\delta$ has cofinality $\ka$. 
In this case, there are $2^\ka$ many cofinal branches of $T \restriction \delta$. 
The definition of $T$ at this level depends on selecting which cofinal branches of 
$T \restriction \delta$ 
will have upper bounds in $T(\delta)$. 
It is this part of the construction which will vary in later constructions.

Consider any $x \in T \restriction \delta$ and $q \in \qk$ such that $\pi(x) \qo q$. 
Using Lemma 1.3(1) and inductive hypotheses (4) and (5) as in the previous case, 
recursively construct a chain above $x$ of length $\ka$ 
whose elements have heights cofinal in $\delta$ 
and whose values under $\pi$ are below $q$. 
Let $b(x,q)$ be the downward closure of this chain. 
Then $b(x,q)$ is a cofinal branch of $T \restriction \delta$. 
Using the fact that each node in $T \restriction \delta$ 
has $\ka$ many immediate successors, it is easy to arrange the function which maps 
$(x,q)$ to $b(x,q)$ to be injective. 
Let the nodes of $T(\delta)$ be the ordinals in the interval 
$[\ka \cdot \delta, \ka \cdot (\delta +1) )$. 
For each cofinal branch of the form $b(x,q)$, 
place one node above $b(x,q)$ and map it under $\pi$ to $q$. 
An argument similar to that in the previous case shows that 
the inductive hypotheses are maintained. 
Note that our definition in this case of cofinality $\ka$ 
gives us a stronger version of property (5), where 
``$\pi(y) \le_{\qk} q$'' is replaced by ``$\pi(y) = q$.''

This completes the construction of $T$ and $\pi$. 
Observe that since $\qk$ has size $\ka$, $T$ is special, and it is clearly normal 
and $\ka$-complete.

Let us make an additional observation about $T$ which we will need later. 
We claim that there exists an antichain $A \subseteq T$ such that 
$\{ \h_T(y) : y \in A \} = \kp \cap \cof(\ka)$. 
Namely, fix any $q \in \qk$ such that $\pi(0) \qo q$. 
Using the strengthened version of property (5) in the case of cofinality $\ka$, 
for each $\beta \in \kp \cap \cof(\ka)$ choose $y_\beta \in T(\beta)$ 
such that $\pi(y_\beta) = q$. 
Let $A := \{ y_\beta : \beta \in \kp \cap \cof(\ka) \}$. 
By property (2) of $T$, $A$ is an antichain.

\smallskip

We now turn to the combinatorial principles which we will use to construct 
non-club-isomorphic Aronszajn trees. 

Let $\ka$ be an infinite cardinal and $S \subseteq \kp$ a stationary set. 
Recall that $\Diamond(S)$ is the statement that there exists a sequence 
$\langle s_\alpha : \alpha \in S \rangle$, where each $s_\alpha \subseteq \alpha$, 
satisfying that for any set $X \subseteq \kp$ 
the set $\{ \alpha \in S : X \cap \alpha = s_\alpha \}$ is stationary. 
And $\Diamond^*(S)$ is the statement that there exists a sequence 
$\langle \mathcal S_\alpha : \alpha \in S \rangle$, where each $\mathcal S_\alpha \subseteq P(\alpha)$ 
has size $\ka$, satisfying that for any set $X \subseteq \kp$ there is a club $D \subseteq \kp$ 
such that for all $\alpha \in D \cap S$, $X \cap \alpha \in \mathcal S_\alpha$. 
We have that $\Diamond^*(S)$ implies $\Diamond(S)$ and $\Diamond(S)$ implies $2^\ka = \kp$.

\begin{thm}[Shelah \cite{diamonds}]
	Let $\ka$ be an uncountable cardinal and assume that $2^\ka = \kp$. 
	Then $\Diamond(\kp)$ holds.
\end{thm}

In our Aronszajn tree constructions we will use the following 
combinatorial principle of Brodsky-Rinot \cite{rinot22}.

\begin{definition}
	Let $\ka$ be a regular cardinal and $\mathcal S$ a family of stationary 
	subsets of $\kp \cap \cof(\ka)$. 
	Then $\boxtimes_\ka^*(\mathcal S)$ is the statement that there exists a sequence 
	$\langle C_\delta : \delta \in \kp \cap \cof(\ka) \rangle$, 
	where each $C_\delta$ is a club subset of 
	$\delta$ of order type $\ka$, satisfying that for any 
	unbounded set $X \subseteq \kp$ and any $S \in \mathcal S$, 
	there are stationarily many $\delta \in S$ 
	for which the set $(C_\delta \cap X) \setminus \lim(C_\delta)$ is cofinal in $\delta$. 
	In the case $\mathcal S = \{ S \}$ is a singleton, $\boxtimes_\ka^*(S)$ abbreviates 
	$\boxtimes_\ka^*(\{ S \})$.
\end{definition}

Note that if $\mathcal S' \subseteq \mathcal S$, then 
$\boxtimes_\ka^*(\mathcal S)$ implies $\boxtimes_\ka^*(\mathcal S')$.

The statement $\boxtimes_\ka^*(\mathcal S)$ is a special case of the very general \emph{proxy principle} 
which was introduced in \cite{rinot22} and used to construct a wide variety 
of Suslin trees. 
In particular, $\boxtimes_\ka^*(\mathcal S)$ is the instance 
$\text{P}^-(\kappa^+, 2,\mathrel{_{\ka}{\sqsubseteq}}, 1, \mathcal S,  2,1,\mathcal E_\kappa)$ of 
\cite[Definition 1.5]{rinot22}. 
For any stationary set $S \subseteq \kp \cap \cof(\ka)$, 
$\Diamond(S)$ implies $\boxtimes_\ka^*(S)$ by \cite[Theorem 5.1(2)]{rinot22}, and 
$\Diamond^*(\kp \cap \cof(\ka))$ implies 
$\boxtimes_\ka^*(\text{NS}^+ \restriction (\kp \cap \cof(\ka)))$ 
by \cite[Theorem 4.35]{rinot23}, 
where $\text{NS}^+ \restriction T$ is the collection of all stationary subsets of $T$ for any 
stationary set $T \subseteq \kp$.

Rather than using the principle $\boxtimes_\ka^*(\mathcal S)$ directly in our 
Aronszajn tree constructions, we will use consequences of $\boxtimes_\ka^*(\mathcal S)$ 
which are a kind of combination of 
diamond and club guessing principles. 
These consequences are described in the next two lemmas, 
whose proofs follow easily from the methods of \cite{rinot22}.

\begin{lemma}
	Let $\ka$ be a regular uncountable cardinal satisfying $2^\ka = \kp$, 
	and let $\langle s_\alpha : \alpha < \kp \rangle$ be a $\Diamond(\kp)$-sequence. 
	Let $S$ be a stationary subset of $\kp \cap \cof(\ka)$ satisfying $\boxtimes_\ka^*(S)$, 
	and choose a sequence $\langle C_\delta : \delta \in \kp \cap \cof(\ka) \rangle$ which 
	witnesses $\boxtimes_\ka^*(S)$. 
	Write each $C_\delta = \{ \alpha_{\delta,i} : i < \ka \}$ in increasing order. 
	Then for any club $C \subseteq \kp$ and any set $X \subseteq \kp$, 
	there are stationarily many $\delta \in S$ such that 
	$\sup\{ i \in \ka \cap \Succ : \alpha_{\delta,i} \in C, \ X \cap \alpha_{\delta,i} = 
	s_{\alpha_{\delta,i}} \} = \ka$.
\end{lemma}

\begin{proof}
	Let $C \subseteq \kp$ be a club and $X \subseteq \kp$ a set. 
	Define $U := \{ \alpha \in \kp : X \cap \alpha = s_\alpha \}$, which is 
	stationary by the diamond property. 
	So $U \cap C$ is stationary. 
	In particular, $U \cap C$ is an unbounded subset of $\kp$. 
	By the choice of $\langle C_\delta : \delta \in \kp \cap \cof(\ka) \rangle$, 
	there are stationarily many $\delta \in S$ for which the set 
	$(C_\delta \cap U \cap C) \setminus \lim(C_\delta)$ is cofinal in $\delta$. 
	For any such $\delta$, 
	there are cofinally many $i < \ka$ such that $i \in \Succ$ and 
	$\alpha_{\delta,i} \in U \cap C$, and hence by the choice of $U$, 
	$X \cap \alpha_{\delta,i} = s_{\alpha_{\delta,i}}$.
\end{proof}

The proof of the next lemma is similar and we leave it for the interested reader.

\begin{lemma}
	Let $\ka$ be a regular uncountable cardinal satisfying $2^\ka = \kp$, 
	and let $\langle s_\alpha : \alpha < \kp \rangle$ be a $\Diamond(\kp)$-sequence. 
	Let $S$ be a stationary subset of $\kp \cap \cof(\ka)$ satisfying 
	$\boxtimes_\ka^*(\text{NS}^+ \restriction S)$, and choose a sequence 
	$\langle C_\delta : \delta \in \kp \cap \cof(\ka) \rangle$ which witnesses 
	$\boxtimes_\ka^*(\text{NS}^+ \restriction S)$. 
	Write each $C_\delta = \{ \alpha_{\delta,i} : i < \ka \}$ in increasing order. 
	Then for any club $C \subseteq \kp$ and any set $X \subseteq \kp$, 
	there exists a club $D \subseteq \kp$ such that for all $\delta \in D \cap S$, 
	$\sup\{ i \in \ka \cap \Succ : \alpha_{\delta,i} \in C, \ X \cap \alpha_{\delta,i} = 
	s_{\alpha_{\delta,i}} \} = \ka$.
\end{lemma}

\section{Distinct Aronszajn Trees}

In this section we show how to use the proxy principle to construct non-club-isomorphic 
Aronszajn trees. 
This construction can be thought of as a warm up for the main result in Section 3.

\begin{thm}
	Let $\ka$ be a regular uncountable cardinal such that $\ka^{<\ka} = \ka$ and $2^\ka = \kp$. 
	Assume that $\boxtimes^*_\kappa(\kp \cap \cof(\ka))$ holds. 
	Then there exist two normal $\ka$-complete $\rk$-embeddable 
	$\kp$-Aronszajn trees which are not club isomorphic. 
	One of the two trees $U$ satisfies that for some antichain $A \subseteq U$, 
	$\{ \h_U(x) : x \in A \} = \kp \cap \cof(\ka)$, 
	whereas the other tree $T$ has the property 
	that there does not exist an antichain $B \subseteq T$ 
	and a club $C \subseteq \kp$ such that 
	$C \cap \cof(\ka) \subseteq \{ \h_T(y) : y \in B \}$. 
\end{thm}

In particular, the conclusion of the theorem follows from 
$\ka^{<\ka} = \ka$ and $\Diamond(\kp \cap \cof(\ka))$ (this assumption is 
strictly stronger than that of Theorem 2.1; see the proof of Example 1.30 and footnote 15 
of \cite{rinot22}). 
An alternative proof under the assumptions of Theorem 2.1 
of the existence of two non-club-isomorphic normal $\ka$-complete $\rk$-embeddable 
$\kp$-Aronszajn trees based on a result of 
Brodsky-Rinot \cite{rinot29} is sketched at the end of the section.

\begin{lemma}
	Let $\lambda$ be a regular uncountable cardinal. 
	Assume that $T$ and $U$ are club-isomorphic $\lambda$-trees. 
	If $A \subseteq U$ is an antichain such that $S := \{ \h_U(x) : x \in A \}$ 
	is a stationary subset of $\lambda$, then there exists a club 
	$C \subseteq \lambda$ and an antichain $B \subseteq T$ such that 
	$\{ \h_T(y) : y \in B \} = S \cap C$.
\end{lemma}

\begin{proof}
	Fix a club isomorphism $f : U \restriction C \to T \restriction C$. 
	Let $A' := \{ x \in A : \h_U(x) \in C \}$. 
	Note that $\{ \h_U(x) : x \in A' \} = S \cap C$. 
	Since $f$ is an isomorphism and $A' \subseteq \dom(f)$, 
	$B := f[A']$ is an antichain of $T$. 
	And $\{ \h_T(y) : y \in B \} = 
	\{ \h_T(f(x)) : x \in A' \} = \{ \h_U(x) : x \in A' \} = S \cap C$.
\end{proof}

We now prove Theorem 2.1. 
Fix a regular cardinal $\ka$ 
such that $\ka^{<\ka} = \ka$ and $2^\ka = \kp$. 
Assume that $\boxtimes^*_\kappa(\kp \cap \cof(\ka))$ holds. 
Fix sequences $\langle C_\delta : \delta \in \kp \cap \cof(\ka) \rangle$ and 
$\langle s_\alpha : \alpha < \kp \rangle$ satisfying the description given in Lemma 1.9. 
Our goal is to produce normal $\ka$-complete $\rk$-embeddable 
$\kp$-Aronszajn trees $T$ and $U$ which are not club isomorphic.

By the construction in Section 1, we can fix a 
normal $\ka$-complete $\kp$-Aronszajn tree $U$ which is $\qk$-embeddable 
and satisfies that for some antichain $A \subseteq U$,  
$\{ \h_U(x) : x \in A \} = \kp \cap \cof(\ka)$. 
Then $U$ is also $\rk$-embeddable.  
We will build a normal $\ka$-complete $\rk$-embeddable 
$\kp$-Aronszajn tree $T$ satisfying that there does not exist an antichain $B \subseteq T$ 
and a club $C \subseteq \kp$ such that 
$C \cap \cof(\ka) \subseteq \{ \h_T(y) : y \in B \}$. 
Lemma 2.2 then implies that $T$ and $U$ are not club isomorphic, which completes the proof.

We construct $T$ together with a function $\pi : T \to \rk$ by recursion. 
We will maintain the following properties:
\begin{enumerate}
	\item $T(0)$ consists of $0$, $T(1)$ consists of the ordinals in the interval 
	$[1,\ka \cdot 2)$, and for each $1 < \alpha < \kp$, $T(\alpha)$ consists of 
	the ordinals in the interval $[\ka \cdot \alpha, \ka \cdot (\alpha+1) )$;
	\item $x <_{T} y$ implies $\pi(x) \ro \pi(y)$, for all $x, y \in T$;
	\item for each $x \in T$, the restriction of $\pi$ to the immediate successors of $x$ 
	is a bijection from that set onto the set $\{ q \in \qk : \pi(x) \ro q \}$;
	\item if $\delta \in \kp \cap \cof(<\! \ka)$, then 
	every cofinal branch $b$ of $T \restriction \delta$ has a unique 
	upper bound $y$ in $T(\delta)$, and $\pi(y) = \sup \{ \pi(x) : x \in b \}$;
	\item for all $x \in T$, $\beta < \kp$, and $q \in \qk$ with
	$\h_{T}(x) < \beta$ and $\pi(x) \ro q$, 
	there exists $y \in T(\beta)$ above $x$ such that $\pi(y) \le_{\rk} q$. 
\end{enumerate}

We will abbreviate the restriction of $\pi$ to $T \restriction \delta$ by $\pi \restriction \delta$ 
for all $\delta < \kp$.

The base case, successor steps, and limit stages of cofinality less than $\ka$ 
are handled in basically the same way as in the construction of the Aronszajn tree in Section 1. 
The only difference is that we are mapping the nodes of $T$ into $\rk$ instead of $\qk$. 
But nodes of successor height will have their values under $\pi$ in $\qk$ by (3), and by 
Lemma 1.5(2) nodes whose heights are limit ordinals of cofinality less than $\ka$ 
also have their values under $\pi$ in $\qk$.
Since $\qk$ is dense in $\rk$, property (5) is easily shown to hold at successor levels and 
at limit levels of cofinality less than $\ka$. 
Thus, the inductive hypotheses hold at these types of levels.

Assume that $\delta < \kp$ has cofinality $\ka$ and we have defined 
$T \restriction \delta$ and $\pi \restriction \delta$ as required. 
We would like to associate to each pair $(x,q)$, where $x \in T \restriction \delta$ 
and $q \in \qk$ with $\pi(x) \ro q$, a cofinal branch $b(x,q)$ of $T \restriction \delta$. 
Then we will add an upper bound $y$ to $b(x,q)$ at level $\delta$ 
and define $\pi(y)$ so that $\pi(y) <_{\rk} q$ and $\pi(z) \ro \pi(y)$ 
for all $z \in b(x,q)$. 
If we succeed in doing this, the inductive hypotheses are clearly maintained. 
Since each node has $\ka$ many immediate successors and there are only $\ka$ many 
pairs $(x,q)$ to handle, it is easy to arrange that the function mapping $(x,q)$ to 
$b(x,q)$ is injective, so we will neglect this point.

Fix $x \in T \restriction \delta$ and $q \in \qk$ with $\pi(x) \ro q$. 
Recall that $C_\delta = \{ \alpha_{\delta,i} : i < \ka \}$ is a club subset of $\delta$. 
Let $i^* < \ka$ be the least ordinal such that $\h_T(x) < \alpha_{\delta,i^*}$. 
We will recursively define sequences $\langle x_i : i < \ka \rangle$ and 
$\langle q_i : i < \ka \rangle$ satisfying:
\begin{enumerate}
	\item for all $i \le i^*$, $x_i = x$ and $q_i = q$;
	\item for all $i^* < i < \ka$, $x_i$ is a node above $x$ 
	on level $\alpha_{\delta,i}$ of $T$ and $q_i \in \qk$;
	\item $q_j \qo q_i$ for all $i^* \le i < j < \ka$;
	\item for all $i, j < \ka$, $\pi(x_i) \ro q_j$.  
\end{enumerate}

Begin by setting $x_i := x$ and $q_i := q$ for all $i \le i^*$. 
Now let $i > i^*$ be given and assume that for all $j < i$ we have defined $x_j$ and $q_j$ as required.

Suppose that $i = j + 1$ is a successor ordinal, and we will define $x_{i}$ and $q_{i}$. 
Consider the following statements:
\begin{itemize}
	\item[(A)] $\alpha_{\delta,i} = \ka \cdot \alpha_{\delta,i}$.
	\item[(B)] $s_{\alpha_{\delta,i}}$ is an antichain of $T \restriction \alpha_{\delta,i}$.
	\item[(C)] there exists some $y \in s_{\alpha_{\delta,i}}$ above $x_j$ in $T$ such that 
	$\pi(y) \ro q_j$.
\end{itemize}
First, assume that these statements are not all true. 
In that case, choose $x_i$ to be any node above $x_j$ 
on level $\alpha_{\delta,i}$ of $T \restriction \delta$ 
such that $\pi(x_i) \ro q_j$ using inductive property (5). 
Then choose $q_{i} \in \qk$ 
such that $\pi(x_i) \ro q_{i} \ro q_j$. 
Secondly, assume that all three statements are true. 
Fix $y$ as in (C). 
Then use inductive property (5) to 
choose $x_i$ above $y$ on level $\alpha_{\delta,i}$ of $T \restriction \delta$ satisfying that 
$\pi(x_i) \ro q_j$. 
Then choose $q_i \in \qk$ such that $\pi(x_i) \ro q_i \ro q_j$.

Suppose that $i$ is a limit ordinal. 
Since $i < \ka$, there exists a unique node $x_i$ on level $\alpha_{\delta,i}$ 
of $T$ which is above $x_j$ for all $j < i$, and $\pi(x_i) = \sup \{ \pi(x_j) : j < i \}$. 
By Lemma 1.5(2), $\pi(x_i)$ is in $\qk$. 
By Lemma 1.3(3), the descending 
sequence $\langle q_j : i^* < j < i \rangle$ does not have a greatest lower bound in $\qk$. 
In particular, $\pi(x_i)$ is a lower bound but not a greatest lower bound of this sequence. 
Thus, we can fix $q_i \in \qk$ such that $\pi(x_i) \qo q_i$ and 
for all $j < i$, $q_i \qo q_j$.

Let $b(x,q)$ be the downward closure in $T \restriction \delta$ of the chain $\{ x_i : i < \ka \}$. 
Then $b(x,q)$ is a cofinal branch of $T \restriction \delta$. 
We place some $y$ above this branch on level $\delta$, and define 
$\pi(y) := \inf \{ q_j : j < \ka \}$. 
This makes sense because the set $\{ q_i : i < \ka \}$ is bounded below by $\pi(x)$ and 
$\rk$ is complete. 
Now for all $i < \ka$ and $j < \ka$, $\pi(x_i) \ro \pi(x_{i+1}) \ro q_j$, and hence 
$\pi(x_i) \ro \inf \{ q_j : j < \ka \} = \pi(y)$. 
It follows that for all $z <_T y$, $\pi(z) \ro \pi(y)$. 
Also, $\pi(y) \ro q$. 

This completes the construction of $T$. 
Let us prove that $T$ is as required. 
Clearly, $T$ is a normal $\ka$-complete $\rk$-embeddable $\kp$-Aronszajn tree. 
We claim that there does not exist an antichain $B \subseteq T$ and a club $C \subseteq \kp$ 
such that 
$C \cap \cof(\ka) \subseteq \{ \h_T(y) : y \in B \}$.

Let $B \subseteq T$ be an antichain and $C \subseteq \kp$ a club. 
For each $q \in \qk$, let $C_q$ be the club set of $\alpha < \kp$ such that for all 
$x \in T \restriction \alpha$, if there exists some $z \in B$ such that 
$x <_T z$ and $\pi(z) \ro q$, then there exists such a $z$ which is in $T \restriction \alpha$. 
Let $D$ be the club of all $\alpha < \kp$ such that $\alpha = \ka \cdot \alpha$.
Define 
$
E := C \cap D \cap \bigcap \{ C_q : q \in \qk \}.
$
Then $E$ is club in $\kp$ since $\qk$ has cardinality $\ka$.

Let $S$ be the set of $\delta \in \kp \cap \cof(\ka)$ such that the set 
$\{ i \in \ka \cap \Succ : \alpha_{\delta,i} \in E, \ B \cap \alpha_{\delta,i} = s_{\alpha_{\delta,i}} \}$ 
is cofinal in $\ka$. 
By our choice of the sequences as in Lemma 1.9, $S$ is stationary. 
Fix $\delta$ in $S \cap E$. 
Then $\delta \in C \cap \cof(\ka)$.

Suppose for a contradiction that 
$C \cap \cof(\ka) \subseteq \{ \h_T(y) : y \in B \}$. 
Then in particular, $\delta \in \{ \h_T(y) : y \in B \}$. 
Fix $y^* \in B$ such that $\h_T(y^*) = \delta$. 
By the construction of $T$, $y^*$ is an upper bound of a branch $b(x,q)$ for some 
$x \in T \restriction \delta$ and $q \in \qk$ with $\pi(x) \ro q$.  
Fix $i^* < \ka$ and sequences $\langle x_i : i < \ka \rangle$ and 
$\langle q_j : j < \ka \rangle$ as described in the definition of $b(x,q)$.

As $\delta \in S$, we can find $i = j + 1 \in \ka \cap \Succ$ greater than $i^*$ such that 
$\alpha_{\delta,i} \in E$ and 
$B \cap \alpha_{\delta,i} = s_{\alpha_{\delta,i}}$. 
Since $\alpha_{\delta,i} \in D$, 
$\alpha_{\delta,i} = \ka \cdot \alpha_{\delta,i}$. 
And $s_{\alpha_{\delta,i}} = B \cap \alpha_{\delta,i}$ is an antichain of $T \restriction \alpha_{\delta,i}$. 
Thus, statements (A) and (B) in the case division of the 
definition of $x_{i}$ and $q_{i}$ are met.

We claim that statement (C) holds as well. 
The node $y^*$ is an element of $B$ which is above $x_j$. 
By definition, $\pi(y^*) = \inf \{ q_i : i < \ka \}$, and in particular, 
$\pi(y^*) \ro q_j$. 
Since $\alpha_{\delta,i} \in C_{q_j}$, 
there exists some $y \in B \cap (T \restriction \alpha_{\delta,i})$ 
above $x_j$ such that $\pi(y) \ro q_j$. 
Then $y \in B \cap \alpha_{\delta,i} = s_{\alpha_{\delta,i}}$, proving (C). 
 
By the definition of $x_{i}$, 
there exists some $y \in s_{\alpha_{\delta,i}} = B \cap \alpha_{\delta,i}$ below $x_{i}$. 
But $x_i <_T y^*$, so $y <_T y^*$.  
This is a contradiction since $B$ is an antichain and both $y$ and $y^*$ are in $B$. 
This completes the proof of Theorem 2.1.

\smallskip

Assaf Rinot has pointed out an alternative proof of the existence of 
non-club-isomorphic normal $\ka$-complete $\rk$-embeddable 
$\kp$-Aronszajn trees under the assumptions of Theorem 2.1. 
We give a sketch of this proof now. 
Under those assumptions, by \cite[Theorem 5.4]{rinot29} there exist $\kp$-Aronszajn 
trees $T_0$ and $T_1$ such that $T_0$ is special and $T_1$ is 
normal and not special. 
These trees are $T(\rho_0)$ and $T(\rho_1)$ associated with walks on ordinals 
using a $C$-sequence of clubs bounded in order type by $\ka$, and assuming 
$\ka^{<\ka} = \ka$, such trees are $\R_\ka$-embeddable. 
Let $T_0'$ be a normal subtree of $T_0$.

Now using the method of the proof of \cite[Proposition 3.7]{jk31}, we can enlarge $T_0'$ and $T_1$ 
to normal $\ka$-complete $\rk$-embeddable $\kp$-Aronszajn trees 
$U_0$ and $U_1$ such that $U_0$ is special and $U_1$ is not special. 
But a special tree cannot be club isomorphic to a non-special tree, since by a 
straightforward argument if a 
tree is special on a club of levels, then it is special.

Note however that this alternative proof does not obviously generalize to construct a large 
pairwise far family of Aronszajn trees, as we will do in the next section.

\section{A large pairwise far family of Aronszajn trees}

We now present the main result of the article. 

\begin{thm}
	Let $\ka$ be a regular uncountable cardinal such that $\ka^{<\ka} = \ka$ and $2^\ka = \kp$. 
	Assume that $\boxtimes^*_\kappa(\text{NS}^+ \restriction (\kp \cap \cof(\ka)))$ holds. 
	Then there exists a pairwise far family of $2^{(\kp)}$ many 
	normal $\ka$-complete $\rk$-embeddable non-special $\kp$-Aronszajn trees. 
	More specifically, for any two trees $T$ and $U$ in the family, there is a stationary 
	set $S_{T,U} \subseteq \kp \cap \cof(\ka)$ such that 
	$U \res S_{T,U}$ is special whereas $T$ has no stationary antichain below $S_{T,U}$.
\end{thm}

In particular, the conclusion of the theorem follows from 
$\ka^{<\ka} = \ka$ and $\Diamond^*(\kp \cap \cof(\ka))$.

Theorem 3.1 will follow quickly from the next theorem, which is of independent interest. 

\begin{thm}
	Let $\ka$ be a regular uncountable cardinal such that $\ka^{<\ka} = \ka$ and $2^\ka = \kp$. 
	Let $S \subseteq \kp \cap \cof(\ka)$ be stationary and assume that 
	$\boxtimes_\ka^*(\text{NS}^+ \restriction S)$ holds. 
	Then there exists a normal $\ka$-complete $\rk$-embeddable non-special 
	$\kp$-Aronszajn tree $T_S$ satisfying:
	\begin{enumerate}
	\item the subtree $T_S \restriction (\kp \setminus S)$ is special;
	\item for any antichain $A \subseteq T_S$, the set 
	$\{ \h_{T_S}(x) : x \in A \} \cap S$ is non-stationary.
	\end{enumerate}
\end{thm}

Property (2) says that $T_S$ has no stationary antichain below $S$. 
Observe that (2) implies that $T_S$ is non-special. 
For if $T_S$ is special, then so is $T_S \res S$. 
But then an application of Lemma 1.1 gives a contradiction with property (2).

In the case $S = \kp \cap \cof(\ka)$, 
our proof of Theorem 3.2 will provide an example of a recursive construction 
of an almost-Suslin tree which is not Suslin. 
This fact is in contrast to other constructions of such trees in the literature 
which come from projections or reduced powers of other trees 
(see, for example, \cite{rinotnew} and \cite{rinotnew2}).

We claim that Theorem 3.1 follows from Theorem 3.2. 
To see this, let 
$\ka$ be a regular uncountable cardinal such that $\ka^{<\ka} = \ka$ and $2^\ka = \kp$. 
Fix a family $\{ S_i : i < 2^{(\kp)} \}$ of stationary subsets of $\kp \cap \cof(\ka)$ 
such that for all distinct $i, j < 2^{(\kp)}$, $S_i \setminus S_j$ is stationary. 
The existence of such a family can be proven by a slight variation of 
the proof of \cite[Proposition 1.1]{treesandlinear}. 
Assume that $\boxtimes^*_\kappa(\text{NS}^+ \restriction (\kp \cap \cof(\ka)))$ holds. 
Then for all $i < 2^{(\kp)}$, 
$\boxtimes^*_\kappa(\text{NS}^+ \restriction S_i)$ holds. 
For each $i < 2^{(\kp)}$, fix a $\kp$-Aronszajn tree $T_i$ satisfying the conclusions of 
Theorem 3.2 for the set $S_i$.

Suppose for a contradiction that $T_i$ and $T_j$ are near for some distinct $i, j < 2^{(\kp)}$. 
Fix downwards closed subtrees $T_i'$ and $T_j'$ of $T_i$ and $T_j$ respectively which are club isomorphic. 
Fix a club isomorphism $f : T_i' \restriction C \to T_j' \restriction C$.

Let $S^* := C \cap (S_j \setminus S_i)$, which is a stationary subset of $\kp$. 
As $S^* \subseteq \kp \setminus S_i$, $T_i \restriction S^*$ is special. 
Therefore, $T_i' \restriction S^*$ is special. 
By Lemma 1.1, fix an antichain $A \subseteq T_i' \restriction S^*$ 
such that $S^{**} := \{ \h_{T_i'}(x) : x \in A \}$ is stationary. 
Note that $S^{**} \subseteq S^* \subseteq C \cap S_j$.

Now $S^* \subseteq C$, so $A \subseteq \dom(f)$. 
Hence, $f[A]$ is an antichain of $T_j' \restriction C$. 
So $f[A]$ is an antichain of $T_j$. 
Also 
	$$
	\{ \h_{T_j}(y) : y \in f[A] \} = \{ \h_{T_j'}(f(x)) : x \in A \} = 
	\{ \h_{T_i'}(x) : x \in A \} = S^{**},
	$$
which is a stationary subset of $S_j$. 
Thus, $f[A]$ is an antichain of $T_j$ for which $\{ \h_{T_j}(y) : y \in f[A] \}$ 
is a stationary subset of $S_j$, contradicting the choice of $T_j$.

It remains to prove Theorem 3.2. 
Assume that $\ka^{<\ka} = \ka$ and $2^\ka = \kp$, and let 
$S \subseteq \kp \cap \cof(\ka)$ be a stationary set such that 
$\boxtimes_\ka^*(\text{NS}^+ \restriction S)$ holds.
Fix sequences $\langle C_\delta : \delta \in S \rangle$ and 
$\langle s_\alpha : \alpha < \kp \rangle$ satisfying the description given in Lemma 1.10.

We construct $T_S$ together with a function $\pi_S : T_S \to \rk$ by recursion. 
The construction of these objects will follow along the lines of the two tree constructions 
from Sections 1 and 2. 
In particular, in defining level $\delta$ of $T_S$ for $\delta \in \kp \cap \cof(\ka)$ 
which is not in $S$, we follow the construction from Section 1, and when $\delta \in S$, 
we follow the construction from Section 2. 
We will maintain the following properties:
\begin{enumerate}
	\item $T_S(0)$ consists of $0$, $T_S(1)$ consists of the ordinals in the interval 
	$[1,\ka \cdot 2)$, and for each $1 < \alpha < \kp$, $T_S(\alpha)$ consists of 
	the ordinals in the interval $[\ka \cdot \alpha, \ka \cdot (\alpha+1) )$; 
	\item $x <_{T_S} y$ implies $\pi_S(x) \ro \pi_S(y)$, for all $x, y \in T_S$;
	\item for each $x \in T_S$, $\pi_S$ restricted to the immediate successors of $x$ 
	is a bijection from that set onto the set $\{ q \in \qk : \pi_S(x) \ro q \}$; 
	\item if $\delta \in \kp \cap \cof(<\! \ka)$, then 
	every cofinal branch $b$ of $T_S \restriction \delta$ has a unique 
	upper bound $y$, and $\pi_S(y) = \sup \{ \pi_S(x) : x \in b \}$;
	\item for all $x \in T_S$, $\beta < \kp$, and $q \in \qk$ with
	$\h_{T_S}(x) < \beta$ and $\pi_S(x) \ro q$, 
	there exists $y \in T_S(\beta)$ above $x$ such that $\pi_S(y) \le_{\rk} q$.
\end{enumerate}

We will abbreviate the restriction of $\pi_S$ to 
$T_S \restriction \delta$ by $\pi_S \restriction \delta$.

As in the construction of Section 2, it suffices to define the tree at levels of cofinality $\ka$, 
the other cases being easy. 
For all $\delta \in \kp \cap \cof(\ka)$ which are not in $S$, 
build a family of distinct cofinal branches $b(x,q)$ 
of $T_S \restriction \delta$, 
for each $x \in T_S \restriction \delta$ and $q \in \qk$ with $\pi_{S}(x) \ro q$, just 
as we did in Section 1. 
Then put an ordinal $y$ above each such branch $b(x,q)$ on level $\delta$ 
and define $\pi_S(y) := q$.

This completes the description of $T_S$ on levels not in $S$. 
Since $\pi_S(x)$ will be in $\qk$ for all $x \in T_S$ of height not in $S$, 
$T_S \restriction (\kp \setminus S)$ will be special.

It remains to define $T_S$ and $\pi_S$ at levels which are in $S$. 
So assume that $\delta \in S$ and $T_S \restriction \delta$ and 
$\pi_S \restriction \delta$ have been defined. 
We will associate to each pair $(x,q)$, where $x \in T_S \restriction \delta$ 
and $q \in \qk$ with $\pi_S(x) \ro q$, a cofinal branch $b(x,q)$ of $T_S \restriction \delta$, 
and then add an upper bound $y$ to $b(x,q)$ on level $\delta$ and define $\pi_S(y)$. 
Since each node of $T_S \restriction \delta$ has $\ka$ many immediate successors, it is easy to 
arrange that the function which maps $(x,q)$ to $b(x,q)$ is injective.

Fix $x \in T_S \restriction \delta$ and $q \in \qk$ such that $\pi_S(x) \ro q$, and we 
will define $b(x,q)$. 
Recall that $C_\delta = \{ \alpha_{\delta,i} : i < \ka \}$ is a club subset of $\delta$. 
Let $i^* < \ka$ be the least ordinal such that $\h_{T_S}(x) < \alpha_{\delta,i^*}$.

We will define sequences $\langle x_i : i < \ka \rangle$ and 
$\langle q_i : i < \ka \rangle$ satisfying:
\begin{enumerate}
	\item for all $i \le i^*$, $x_i = x$ and $q_i = q$;
	\item for all $i^* < i < \ka$, $x_i$ is a node on level $\alpha_{\delta,i}$ of $T_S \restriction \delta$ 
	above $x$;
	\item for all $i < \ka$, $q_i \in \qk$;
	\item $q_j \qo q_i$ for all $i^* \le i < j < \ka$;
	\item for all $i,j < \ka$, $\pi_S(x_i) \ro q_j$.  
\end{enumerate}

Begin by setting $x_i := x$ and $q_i := q$ for all $i \le i^*$. Now let $i < \ka$ be greater than $i^*$ and assume that for all $j < i$ we have defined 
$x_j$ and $q_j$ as required.

\bigskip

\noindent Case 1: Assume that $i = j + 1$ is a successor ordinal. 
Consider the following statements:
\begin{itemize}
\item[(A)] $\alpha_{\delta,i} = \ka \cdot \alpha_{\delta,i}$;
\item[(B)] $s_{\alpha_{\delta,i}}$ is an antichain of $T_S \restriction \alpha_{\delta,i}$;
\item[(C)] there exists $z \in s_{\alpha_{\delta,i}}$ above $x_j$ such that 
$\pi_S(z) \ro q_{j}$.
\end{itemize}
First, assume that at least one of these statements is false. 
Since $\pi_S(x_j) \ro q_j$, we can choose 
$x_{i}$ above $x_j$ on level $\alpha_{\delta,i}$ of $T_S \restriction \delta$ 
such that $\pi_S(x_i) \ro q_j$. 
Then choose $q_{i} \in \qk$ such that 
$\pi_S(x_i) \ro q_{i} \ro q_j$.

Secondly, assume that the three statements are true. 
Fix $z$ as in (C). 
Then fix $x_{i}$ above $z$ on level $\alpha_{\delta,i}$ of $T_S \restriction \delta$ 
such that $\pi_S(x_i) \ro q_j$. 
Finally, choose $q_{i} \in \qk$ such that 
$\pi_S(x_i) \ro q_{i} \ro q_j$. 
This completes the definition of $x_i$ and $q_i$ when $i < \ka$ is a successor ordinal.

\bigskip

\noindent Case 2: Assume that $i < \ka$ 
is a limit ordinal and $x_j$ and $q_j$ have been defined for all $j < i$. 
Let $x_i$ be the least upper bound of $\{ x_j : j < i \}$, which is on 
level $\alpha_{\delta,i}$ of $T_S \restriction \delta$. 
Since $\cf(\alpha_{\delta,i}) < \ka$, $\pi_S(x_i) \in \qk$. 
Let $q_i^*$ be the greatest lower bound of $\{ q_j : j < i \}$. 
Then $\pi_S(x_i) \le_{\rk} q_i^*$. 
By Lemma 1.3(3), $q_i^* \notin \qk$, so $\pi_S(x_i) \ro q_i^*$. 
Choose $q_i \in \qk$ with $\pi_S(x_i) \ro q_i \ro q_i^*$.

\bigskip

This completes the construction of the sequences $\langle x_i : i < \ka \rangle$ 
and $\langle q_i : i < \ka \rangle$. 
Define $b(x,q)$ to be the downward closure of the chain $\{ x_i : i < \ka \}$. 
Note that $b(x,q)$ is a cofinal branch of $T_S \restriction \delta$. 
Now put some $y$ above the branch $b(x,q)$ on level $\delta$ of $T_S$ 
and define $\pi_S(y) := \inf \{ q_i : i < \ka \}$. 
Then $\pi_S(y) \ro q_i$ for all $i < \ka$, and hence $\pi_S(y) \ro q$.

\bigskip

This completes the construction of $T_S$ and $\pi_S$. 
It is clear that $T_S$ is a normal $\ka$-complete $\rk$-embeddable $\kp$-Aronszajn tree 
satisfying that $T_S \restriction (\kp \setminus S)$ is special. 
It remains to prove that for any antichain $A \subseteq T_S$, the set 
$\{ \h_{T_S}(x) : x \in A \} \cap S$ is non-stationary.

Let $A \subseteq T_S$ be an antichain. 
For each $q \in \qk$, define $C_q$ to be the set of all $\delta < \kp$ satisfying that 
for all $x \in T_S \restriction \delta$, 
if there exists some $y \in A$ above $x$ such that 
$\pi_S(y) \ro q$, then there exists such a $y$ in $T_S \restriction \delta$.
Let $C := \bigcap \{ C_q : q \in \qk \} \cap \{ \alpha \in \ka : \kappa \cdot \alpha = \alpha \}$. 
Since $\qk$ has size $\ka$, $C$ is a club subset of $\kp$.

By the property described in Lemma 1.10, 
fix a club $D \subseteq \kp$ 
such that for all $\delta \in D \cap S$, the set 
$\{ i \in \ka \cap \Succ : \alpha_{\delta,i} \in C, 
\ A \cap \alpha_{\delta,i} = s_{\alpha_{\delta,i}} \}$ is cofinal in $\ka$.

We claim that 
$\{ \h_{T_S}(x) : x \in A \} \cap S \cap D = \emptyset$, 
which will complete the proof. 
Suppose for a contradiction that $\delta$ is in this intersection. 
Fix $y \in A$ such that $\h_{T_S}(y) = \delta$. 
By the definition of level $\delta$ of $T_S$, there exists some $x \in T_S \restriction \delta$ 
and $q \in \qk$ such that $y$ is the unique upper bound of $b(x,q)$ on level $\delta$.

Let $i^*$, $\langle x_i : i < \ka \rangle$, and $\langle q_i : i < \ka \rangle$ be 
as in the definition of $b(x,q)$. 
Recall that $\pi_S(y) \ro q_i$ for all $i < \ka$. 
Fix $i = j + 1 \in \ka \cap \Succ$ greater than $i^*$ such that 
$\alpha_{\delta,i} \in C$ and $A \cap \alpha_{\delta,i} = s_{\alpha_{\delta,i}}$.

Now $y \in A$, $x_{j} <_{T_S} y$, and $\pi_S(y) \ro q_{j}$. 
Since $\alpha_{\delta,i} \in C \subseteq C_{q_j}$, 
there exists 
$
z \in A \cap (T_S \restriction \alpha_{\delta,i}) = A \cap \alpha_{\delta,i} = s_{\alpha_{\delta,i}} 
$
such that 
$x_{j} <_{T_S} z$ and $\pi_S(z) \ro q_{j}$. 
By definition, $x_{i}$ is above such a $z$. 
So $z <_{T_S} x_i <_{T_S} y$, and hence $z <_{T_S} y$. 
This is a contradiction because both $y$ and $z$ are in $A$ and $A$ is an antichain. 
This completes the proof of Theorem 3.1.

\smallskip

Finally, we point out that the assumptions of Theorem 3.1 
hold in a variety of models in addition to one satisfying $\Diamond^*(\kp \cap \cof(\ka))$.  
Assuming $2^{\ka} = \kp$ and $\ka^{<\ka} = \ka$, \cite[Theorem 6.1]{rinot23} shows that 
$\boxtimes_\ka^*(\text{NS}^+ \restriction (\kp \cap \cof(\ka)))$ holds 
in any of the following generic extensions:
\begin{enumerate}
	\item any generic extension by $\add(\ka,1)$;
	\item if $\chi > \kappa$ is strongly inaccessible, any generic extension by 
	a $(<\!\kappa)$-distributive $\chi$-c.c. forcing which collapses $\chi$ to become $\kp$;
	\item any generic extension by a 
	$\kp$-c.c.\! forcing of size $\le\kappa^+$ which preserves the regularity of $\kappa$ 
	and is not ${}^\kappa\kappa$-bounding.
\end{enumerate}
In particular, in any such generic extension 
there exists a pairwise far family of $2^{(\kp)}$ many 
normal $\ka$-complete $\rk$-embeddable $\kp$-Aronszajn trees.

\bibliographystyle{plain}
\bibliography{paper38}

\end{document}